\documentclass{amsart}
\usepackage[british]{babel}
\usepackage{amsthm,amssymb,mathtools}
\usepackage{enumitem}
\usepackage[latin1]{inputenc}

\usepackage{url}
\usepackage{hyperref}
\usepackage{tikz-cd}

\newtheorem{theorem}{Theorem}
\numberwithin{theorem}{section}
\newtheorem{corollary}[theorem]{Corollary}
\newtheorem{lemma}[theorem]{Lemma}

\theoremstyle{definition}
\newtheorem{definition}[theorem]{Definition}

\newcommand{\rca}{\mathbf{RCA}}
\newcommand{\aca}{\mathbf{ACA}}
\newcommand{\supp}{\operatorname{supp}}
\newcommand{\otp}{\operatorname{otp}}

\newcommand{\en}{\operatorname{en}}

\begin{document}


\title[A note on ordinal exponentiation and derivatives of normal functions]{A note on ordinal exponentiation and\\ derivatives of normal functions}

\author{Anton Freund}
\address{Fachbereich Mathematik, Technische Universit\"at Darmstadt, Schlossgartenstr.~7, 64289~Darmstadt, Germany}
\email{freund@mathematik.tu-darmstadt.de}

\begin{abstract}
Michael Rathjen and the present author have shown that $\Pi^1_1$-bar induction is equivalent to (a suitable formalization of) the statement that every normal function has a derivative, provably in $\aca_0$. In this note we show that the base theory can be weakened to $\rca_0$. Our argument makes crucial use of a normal function~$f$ with $f(\alpha)\leq 1+\alpha^2$ and $f'(\alpha)=\omega^{\omega^\alpha}$. We will also exhibit a normal function $g$ with $g(\alpha)\leq 1+\alpha\cdot 2$ and~$g'(\alpha)=\omega^{1+\alpha}$.
\end{abstract}

\maketitle
{\let\thefootnote\relax\footnotetext{\copyright~2020 The Authors. \emph{Mathematical Logic Quarterly} published by Wiley-VCH GmbH.\\
This is the accepted version of a publication in \emph{Mathematical Logic Quarterly} 66:3 (2020)~326-335. Please cite the official journal publication, which is an open access article under the terms of the Creative Commons Attribution-NonCommercial-NoDerivs License.}

\section{Introduction}

A function $f$ from ordinals to ordinals is called normal if it is strictly increasing and continuous at limit stages. The latter means that we have $f(\lambda)=\sup_{\alpha<\lambda}f(\alpha)$ whenever~$\lambda$ is a limit ordinal. For any normal function~$f$, the class $\{\alpha\,|\,f(\alpha)=\alpha\}$ of fixed points is closed and unbounded. The strictly increasing enumeration of these fixed points is itself a normal function, which is called the derivative $f'$ of~$f$. In the present paper we will only be concerned with normal functions that map countable ordinals to countable ordinals. These play an important role in proof theory and have interesting computability theoretic properties (see~\cite{schuette77,marcone-montalban}).

In many investigations of normal functions, ordinal exponentiation is presupposed as a starting point. Most notably, the first function in the Veblen hierarchy is usually defined as $\varphi_0(\alpha)=\omega^\alpha$ (see e.\,g.~\cite{schuette77}). This makes a lot of sense in the context of ordinal notation systems, since a non-zero ordinal is of the form $\omega^\alpha$ if and only if it is closed under addition. On the other hand, ordinal exponentiation does itself presuppose certain set existence principles, as the following result from reverse mathematics shows (see below for an introduction):

\begin{theorem}[J.-Y.~Girard~\cite{girard87}, J.~Hirst~\cite{hirst94}]\label{thm:girard-hirst}
 The following are equivalent over the base theory $\rca_0$:
 \begin{itemize}
  \item arithmetical comprehension (i.\,e.~the principal axiom of $\aca_0$),
  \item if $(X,<_X)$ is a well-order, then so is
  \begin{equation*}
   2^X=\{\langle x_1,\dots,x_n\rangle\,|\,x_1,\dots,x_n\in X\text{ and }x_n<_X\dots<_X x_1\}
  \end{equation*}
  with the lexicographic order.
 \end{itemize}
\end{theorem}

Note that elements of $2^X$ correspond to ordinals in base-$2$ Cantor normal form. In particular, $2^X$ has order type $2^\alpha$ (as usually defined in ordinal arithmetic) if $X$ has order type~$\alpha$. The theorem is also valid with base $\omega$ (recall $\omega^\alpha=2^{\omega\cdot\alpha}$), but base $2$ will have technical advantages in the following.

Theorem~\ref{thm:girard-hirst} has been formulated as a result of reverse mathematics. In this research program one investigates implications between foundational and mathematical principles that can be expressed in the language of second order arithmetic. Implications and equivalences are proved over some weak base theory, most commonly in~$\rca_0$. The latter can handle primitive recursive constructions relative to an oracle. More specifically, the acronym $\rca_0$ alludes to the recursive comprehension axiom, which allows to form $\Delta^0_1$-definable subsets of~$\mathbb N$. Furthermore, $\rca_0$ allows induction for induction formulas of complexity $\Sigma^0_1$. For a detailed introduction to reverse mathematics we refer to~\cite{simpson09}.

The case of Theorem~\ref{thm:girard-hirst} shows that it is relatively straightforward to consider specific normal functions in reverse mathematics. It is considerably more difficult to express statements that quantify over all normal functions, or at least over a sufficiently rich class. In order to do so, one needs a general representation of normal functions by subsets of the natural numbers. Such a representation is possible via J.-Y.~Girard's~\cite{girard-pi2} notion of dilator and related work by P.~Aczel~\cite{aczel-phd,aczel-normal-functors}. Full details have been worked out in~\cite[Section~2]{freund-rathjen_derivatives}; we will recall them as they become relevant for the present paper. Relative to the representation of normal functions in second order arithmetic, M.~Rathjen and the present author have shown that the following are equivalent over $\aca_0$ (see~\cite[Theorem~5.9]{freund-rathjen_derivatives}):
\begin{enumerate}
 \item Every normal function has a derivative.
 \item The principle of $\Pi^1_1$-bar induction (also called transfinite induction) holds.
\end{enumerate}
Considering the proof given in~\cite{freund-rathjen_derivatives}, we see that the implication from (1) to (2) uses arithmetical comprehension (in the form of the Kleene normal form theorem, cf.~\cite[Lemma~V.1.4]{simpson09}). The proof that (2) implies (1) is carried out in $\rca_0$. In any case, a result of J.~Hirst~\cite{hirst99} shows that (2) implies arithmetical comprehension (the author is grateful to E.~Frittaion for pointing this out). To establish the equivalence between (1) and (2) over $\rca_0$ it remains to show that~(1) implies arithmetical comprehension as well. This is the main result of the present paper.

Concerning terminology, we use ``bar induction" and ``transfinite induction" synonymously, since this coincides with the usage in other papers on the reverse mathematics of well ordering principles (cf.~e.\,g.~\cite{rathjen-model-bi}). At the same time, we appreciate that bar induction is a conceptually different notion in constructive mathematics.

In the rest of this introduction we sketch the proof that statement (1) above implies arithmetical comprehension. Since we have not yet explained the representation of normal functions in second order arithmetic, the following argument will be rather informal. Formal versions of all claims will be established in the following sections. The idea of the proof is to construct a normal function $f$ such that the following holds for any ordinal $\alpha$ (where $f'$ is the derivative of $f$):
\begin{enumerate}[label=(\roman*)]
 \item We have $f(\alpha)\leq 1+\alpha^2\leq(1+\alpha)^2$.
 \item We have $2^\alpha\leq f'(\alpha)$.
\end{enumerate}
Part~(i) is supposed to ensure that $\rca_0$ recognizes $f$ as a normal function (since it proves that~$(1+\alpha)^2$ is well-founded for any well-order~$\alpha$). Invoking~(1) from above, we obtain access to the well-founded values $f'(\alpha)$ of the derivative. The inequality in~(ii) corresponds to an order embedding of $2^\alpha$ into $f'(\alpha)$, which witnesses that $2^\alpha$ is also well-founded. By Theorem~\ref{thm:girard-hirst} this yields arithmetical comprehension.

Let us now show how clauses (i) and (ii) can be satisfied: Working in a sufficiently strong set theory, the required function $f$ can be described by
\begin{equation*}
f(\alpha)=1+\sum_{\gamma<\alpha}(1+\gamma).
\end{equation*}
More formally, this infinite sum corresponds to the recursive clauses
\begin{align*}
f(0)&=1,\\
f(\alpha+1)&=f(\alpha)+1+\alpha,\\
f(\lambda)&=\textstyle\sup_{\alpha<\lambda}f(\alpha)\qquad\text{for $\lambda$ limit},
\end{align*}
which immediately reveal that $f$ is normal. It might appear more natural to set $f(\alpha+1)=f(\alpha)+\alpha$ in the successor case (at least for $\alpha>0$), but the summand~$1$ will be crucial in the following sections. A straightforward induction on $\alpha$ shows that we have $f(\alpha)\leq1+\alpha^2$. The inequality $2^\alpha\leq f'(\alpha)$ is also proved by induction on $\alpha$: In view of $1=f(0)\leq f'(0)$ the claim holds for $\alpha=0$. In case $\alpha\neq 0$ we have
\begin{equation*}
2^\alpha=\sup\{2^\beta+\gamma\,|\,\beta<\alpha\text{ and }\gamma<2^\beta\}.
\end{equation*}
Given $\beta<\alpha$ and $\gamma<2^\beta$, the induction hypothesis yields
\begin{multline*}
2^\beta+\gamma< f'(\beta)+2^\beta\leq f(f'(\beta))+1+f'(\beta)={}\\
{}=f(f'(\beta)+1)\leq f(f'(\beta+1))=f'(\beta+1)\leq f'(\alpha),
\end{multline*}
which completes the induction step. When we formalize the proof, we will see that the use of transfinite induction can be avoided, which may be somewhat surprising.

The bound $2^\alpha\leq f'(\alpha)$ suffices to lower the base theory of~\cite[Theorem~5.9]{freund-rathjen_derivatives}, but it is not optimal: In the last section of this note we will establish~$f'(\alpha)=\omega^{\omega^\alpha}$. In ordinal arithmetic one is particularly interested in the function $\alpha\mapsto\omega^\alpha$, which is the most common starting function for the Veblen hierarchy (see e.\,g.~\cite{schuette77}). In this context it is natural to ask whether $\alpha\mapsto\omega^\alpha$ itself is the derivative of some normal function. This is not the case: For any normal function~$g$ we see that $g(0)=0$ implies $g'(0)=0<\omega^0$ while $g(0)>0$ implies $g(1)>1$ and then $g'(0)>1=\omega^0$. However, this value is the only obstruction: We will exhibit a normal function $g$ with $g(\alpha)\leq 1+\alpha\cdot 2$ and $g'(\alpha)=\omega^{1+\alpha}$.

\section{A normal function justified by recursive comprehension}

In the present section we recall how normal functions can be represented in second order arithmetic; further explanations and full details of all missing proofs can be found in~\cite[Section~2]{freund-rathjen_derivatives}. We then apply this representation to the normal function~$f$ that has been considered in the introduction.

To find a representation of normal functions, we need to understand how they can be determined by a countable amount of information. Clearly, normal functions are not determined by their values on some fixed countable set of arguments. This suggests to extend our functions into objects with more internal structure. For this purpose it is convenient to use some notions from category theory, namely those of category, functor and natural transformation. We will not need anything that goes beyond these basic concepts; an introduction to category theory (which is much too comprehensive for our purpose) can be found in~\cite{maclane-working}.

The category of linear orders consists of the linear orders as objects and the embeddings (strictly increasing functions) as morphisms. We will be particularly interested in the finite orders $n=\{0,\dots,n-1\}$ (with the usual order relation). These orders and the embeddings between them form the category of natural numbers. Girard's~\cite{girard-pi2} idea was to consider particularly uniform functors from linear orders to linear orders; those functors that preserve well-foundedness are called dilators. Due to their uniformity, dilators are essentially determined by their restrictions to the category of natural numbers. Provided that these restrictions are countable, they can be used to represent normal functions in reverse mathematics.

In order to describe the uniformity property of dilators, we consider the finite subset functor on the category of sets, which is given by
\begin{align*}
[X]^{<\omega}&=\text{``the set of finite subsets of $X$"},\\
[f]^{<\omega}(a)&=\{f(x)\,|\,x\in a\},
\end{align*}
where the second clause refers to $f:X\to Y$ and $a\in[X]^{<\omega}$. The cardinality of a finite set~$a$ will be denoted by~$|a|\in\mathbb N$ (which can in turn stand for the finite order~$\{0,\dots,|a|-1\}$). The following is essentially due to Girard~\cite{girard-pi2}; we refer to~\cite[Remark~2.2.2]{freund-thesis} for a detailed comparison with his original definition.

\begin{definition}[$\rca_0$]\label{def:prae-dil}
A prae-dilator consists of
\begin{enumerate}[label=(\roman*)]
\item a functor $T$ from natural numbers to linear orders, such that each order $T(n)=(T(n),<_{T(n)})$ has field $T(n)\subseteq\mathbb N$, and
\item a natural transformation $\supp:T\Rightarrow[\cdot]^{<\omega}$ that satisfies the following support condition: Each $\sigma\in T(n)$ lies in the image of the embedding~$T(\en_\sigma)$, where $\en_\sigma:|\supp_n(\sigma)|\to n$ is the strictly increasing enumeration of the set $\supp_n(\sigma)\subseteq n=\{0,\dots,n-1\}$.
\end{enumerate}
\end{definition}

In second order arithmetic, the function $n\mapsto T(n)$ can be represented by the set $T^0=\{(n,\sigma)\,|\,\sigma\in T(n)\}$; the latter is a subset of~$\mathbb N$ if we code each pair by a single number. Officially, this turns $\sigma\in T(n)$ into an abbreviation for $(n,\sigma)\in T^0$, which can be expressed by a formula that is $\Delta^0_1$ in~$\rca_0$. Assuming that finite sets and functions are coded by natural numbers, the action $f\mapsto T(f)$ on morphisms and the map $(n,\sigma)\mapsto\supp_n(\sigma)$ can be represented in the same way.

Above we have mentioned that certain functors on linear orders are determined by their restrictions to the category of natural numbers. Conversely, we now explain how a prae-dilator can be extended into a functor on linear orders. In $\rca_0$ we define
\begin{equation}\label{eq:extend-dil}
D^T(X)=\{\langle a,\sigma\rangle\,|\,a\in[X]^{<\omega}\text{ and }\sigma\in T(|a|)\text{ and }\supp_{|a|}(\sigma)=|a|\}
\end{equation}
for any prae-dilator $T=(T,\supp)$ and any linear order $X$. Informally speaking, the pair $\langle a,\sigma\rangle$ represents the element $T(\en_a)(\sigma)\in T(X)$, where $\en_a:|a|\to X$ is the increasing function with image $a\subseteq X$ (note that $T(\en_a)(\sigma)$ would make sense if~$T$ was defined on all linear orders). Due to the condition $\supp_{|a|}(\sigma)=|a|$, the representation is unique (we would have $a=\supp_X(T(\en_a)(\sigma))$ if $\supp$ was defined beyond the category of natural numbers). In order to define the appropriate order relation on $D^T(X)$, we introduce the following notation: Given an embedding $f:a\to b$ between finite orders, let $|f|:|a|\to|b|$ be the unique function that makes
\begin{equation*}
\begin{tikzcd}
{|a|}\arrow["\cong"]{r}\arrow[swap,"|f|"]{d} & a\arrow["f"]{d}\\
{|b|}\arrow["\cong"]{r} & b
\end{tikzcd}
\end{equation*}
a commutative diagram. We can now stipulate
\begin{equation*}
\langle a_0,\sigma_0\rangle <_{D^T(X)} \langle a_1,\sigma_1\rangle\quad:\Leftrightarrow\quad T(|\iota_0|)(\sigma_0)<_{T(|a_0\cup a_1|)} T(|\iota_1|)(\sigma_1),
\end{equation*}
where $\iota_i:a_i\hookrightarrow a_0\cup a_1$ are the inclusions. It is also possible to turn $D^T(\cdot)$ into a functor and to define natural support functions $\supp_X:D^T(X)\to[X]^{<\omega}$. In particular we can declare that $T$ is a dilator if and only if the order $D^T(X)$ is well-founded for any well-order~$X$ (the two obvious definitions of well-ordering are equivalent over $\rca_0$, see e.\,g.~\cite[Lemma~2.3.12]{freund-thesis}). From the viewpoint of a sufficiently strong set theory, each dilator $T$ gives rise to a function $f_T$ from ordinals to ordinals, with
\begin{equation}\label{eq:induced-function}
f_T(\alpha)=\otp(D^T(\alpha)).
\end{equation}
Here we view $\alpha$ as a linear order and write $\otp(X)$ for the order type of~$X$. We can view $T$ as a representation of the function $f_T$ in second order arithmetic.

It is straightforward to specify a dilator $T$ with $f_T(\alpha)=\alpha+1$. In particular, the function $f_T$ does not need to be normal. The following condition, which was identified by Aczel~\cite{aczel-phd,aczel-normal-functors}, ensures that we are concerned with a normal function:

\begin{definition}[$\rca_0$]\label{def:normal-dil}
A normal (prae-)dilator consists of a (prae-)dilator $T$ and a natural family of embeddings $\mu_n:n\to T(n)$ such that
\begin{equation*}
\sigma<_{T(n)}\mu_n(m)\quad\Leftrightarrow\quad\supp_n(\sigma)\subseteq m=\{0,\dots,m-1\}
\end{equation*}
holds for all $\sigma\in T(n)$ and all $m<n$.
\end{definition}

Note that we necessarily have $\supp_1(\mu_1(0))=1$, since $\supp_1(\mu_1(0))=\emptyset$ would yield $\mu_1(0)<_{T(1)}\mu_1(0)$. This allows us to define $D^\mu_X:X\to D^T(X)$ by
\begin{equation}\label{eq:extend-normal}
D^\mu_X(x)=\langle\{x\},\mu_1(0)\rangle.
\end{equation}
One can show that we have
\begin{equation*}
\langle a,\sigma\rangle<_{D^T(X)}D^\mu_X(x)\quad\Leftrightarrow\quad a\subseteq X\!\restriction\!x=\{x'\in X\,|\,x'<_X x\}.
\end{equation*}
Hence the elements $D^\mu_X(x)$ are cofinal in $D^T(X)$ if $X$ has limit type. In a sufficiently strong set theory one can deduce that $f_T$ is normal (cf.~\cite[Proposition~2.12]{freund-rathjen_derivatives}).

In the introduction we have considered a normal function $f$ with
\begin{equation*}
f(\alpha)=1+\sum_{\gamma<\alpha}(1+\gamma).
\end{equation*}
Our next goal is to construct a normal dilator $F$ that represents this function. Given an order $X$, we write
\begin{equation*}
1+X=\{\bot\}\cup X
\end{equation*}
for the extension of $X$ by a new minimum element~$\bot$. To obtain a functor we map each embedding $h:X\rightarrow Y$ to the embedding $1+h:1+X\to 1+Y$ with
\begin{equation*}
(1+h)(x)=\begin{cases}
\bot & \text{if }x=\bot,\\
h(x) & \text{if }x\in X.
\end{cases}
\end{equation*}
In order to define a dilator $F$ we must specify a linear order $F(n)$ for each finite order~$n=\{0,\dots,n-1\}$. It will later be convenient to have a more general definition, which explains $F(X)$ for any linear order $X$.

\begin{definition}[$\rca_0$]\label{def:F}
For each linear order $X$ we define
\begin{equation*}
F(X)=1+\sum_{x\in 1+X}(1+X)\!\restriction\!x=\{\bot\}\cup\{\langle x,y\rangle\in(1+X)^2\,|\,y<_{1+X}x\}.
\end{equation*}
Note that $F(X)$ contains no pairs of the form $\langle \bot,y\rangle$, since $y<_{1+X}\bot$ must fail. To turn $F(X)$ into a linear order we declare that $\bot$ is minimal and that we have
\begin{equation*}
\langle x_0,y_0\rangle<_{F(X)}\langle x_1,y_1\rangle\quad\Leftrightarrow\quad\begin{cases}
\text{either $x_0<_X x_1$},\\
\text{or $x_0=x_1$ and $y_0<_{1+X}y_1$}.
\end{cases}
\end{equation*}
For an embedding $h:X\to Y$, define $F(f):F(X)\to F(Y)$ by $F(f)(\bot)=\bot$ and
\begin{equation*}
F(h)(\langle x,y\rangle)=\langle h(x),(1+h)(y)\rangle.
\end{equation*}
Each order $X$ gives rise to a function $\supp^F_X:F(X)\to[X]^{<\omega}$ with
\begin{equation*}
\supp^F_X(\bot)=\emptyset\qquad\text{and}\qquad\supp^F_X(\langle x,y\rangle)=\begin{cases}
\{x\} & \text{if }y=\bot,\\
\{x,y\} & \text{if }y\in X.
\end{cases}
\end{equation*}
Finally, we define functions $\mu^F_X:X\to F(X)$ by setting $\mu^F_X(x)=\langle x,\bot\rangle$.
\end{definition}

Note that the relations $\sigma\in F(n)$, $\sigma<_{F(n)}\tau$, $F(h)(\sigma)=\tau$ with $h:n\to m$, $a=\supp^F_n(\sigma)$ and $\sigma=\mu^F_n(m)$ are $\Delta^0_1$-definable in $\rca_0$. Hence the restriction of $F$ to the category of natural numbers exists as a set. It is straightforward to verify that Definitions~\ref{def:prae-dil} and~\ref{def:normal-dil} are satisfied (the condition $y<_{1+X}x$ in the definition of $F(X)$ is crucial for the latter):

\begin{lemma}[$\rca_0$]\label{lem:F-prae-dil}
Restricting Definition~\ref{def:F} to the category of natural numbers yields a normal prae-dilator~$F$.
\end{lemma}

To show that $F$ is a dilator we need to consider the ordered sets $D^F(X)$  from equation~(\ref{eq:extend-dil}). As a preparation, we relate $D^F(X)$ to the order $F(X)$ constructed in Definition~\ref{def:F}. Let us also recall that $\mu^F$ (or rather its restriction to the category of natural numbers) gives rise to a family of functions $D^{\mu^F}_X:X\to D^F(X)$, as defined by equation~(\ref{eq:extend-normal}). For later use, we relate these to the functions $\mu^F_X:X\to F(X)$.

\begin{lemma}[$\rca_0$]\label{lem:iso-reconstruct}
For each order~$X$ we have an isomorphism
\begin{equation*}
\eta_X:D^F(X)\xrightarrow{\cong} F(X)
\end{equation*}
with $\eta_X\circ D^{\mu^F}_X=\mu^F_X$.
\end{lemma}
\begin{proof}
Recall that $D^F(X)$ consists of pairs $\langle a,\sigma\rangle$, where $a$ is a finite suborder of $X$ and $\sigma\in F(|a|)$ satisfies $\supp^F_{|a|}(\sigma)=|a|$. We set
\begin{equation*}
\eta_X(\langle a,\sigma\rangle)=F(\en_a)(\sigma),
\end{equation*}
writing $\en_a:|a|\to X$ for the increasing function with range $a$. It is straightforward to verify that $F$ is an endofunctor on the category of linear orders. Using this fact one can show that $\eta_X$ is an embedding, as in the proof of~\cite[Proposition~2.1]{freund-computable}: Given sets~$Y\subseteq Z$, we agree to write $\iota_Y^Z:Y\hookrightarrow Z$ for the inclusion. Assume that we have $\langle a,\sigma\rangle<_{D^F(X)}\langle b,\tau\rangle$, and note that this amounts to
\begin{equation*}
F(|\iota_a^{a\cup b}|)(\sigma)<_{F(|a\cup b|)} F(|\iota_b^{a\cup b}|)(\tau).
\end{equation*}
Given a finite order~$c$, we write $\en^0_c:|c|\to c$ for the increasing enumeration. For the function $\en_a:a\to X$ from above, the definition of~$|\cdot|$ yields
\begin{equation*}
\en_a=\iota_a^X\circ\en^0_a=\iota_{a\cup b}^X\circ\iota_a^{a\cup b}\circ\en^0_a=\iota_{a\cup b}^X\circ\en^0_{a\cup b}\circ|\iota_a^{a\cup b}|.
\end{equation*}
Since $F(\iota_{a\cup b}^X\circ\en^0_{a\cup b})$ is an embedding, the above inequality implies
\begin{multline*}
\eta_X(\langle a,\sigma\rangle)=F(\en_a)(\sigma)=F(\iota_{a\cup b}^X\circ\en^0_{a\cup b})\circ F(|\iota_a^{a\cup b}|)(\sigma)<_{F(X)}\\
<_{F(X)} F(\iota_{a\cup b}^X\circ\en^0_{a\cup b})\circ F(|\iota_b^{a\cup b}|)(\tau)=\eta_X(\langle b,\tau\rangle).
\end{multline*}
This confirms that $\eta_X$ is an embedding and in particular injective. It remains to show surjectivity. As a representative example, consider $\langle x,y\rangle\in F(X)$ with~$y\neq\bot$. According to Definition~\ref{def:F} we must have $y<_X x$. Hence $a:=\{x,y\}$ has two elements, and the function $\en_a:2\to X$ has values $\en_a(0)=y$ and $\en_a(1)=x$. Since $\sigma:=\langle 1,0\rangle\in F(2)$ satisfies $\supp^F_2(\sigma)=\{0,1\}=2$, we get $\langle a,\sigma\rangle\in D^F(X)$. By construction we have
\begin{equation*}
\eta_X(\langle a,\sigma\rangle)=F(\en_a)(\langle 1,0\rangle)=\langle\en_a(1),(1+\en_a)(0)\rangle=\langle x,y\rangle.
\end{equation*}
A similar argument shows that the image of $\eta_X$ contains $\bot$ and all elements of the form $\langle x,\bot\rangle$. In order to verify the remaining claim we consider $x\in X$ and write $\en_{\{x\}}:1\to X$ for the function with range $\{x\}$. In view of equation~(\ref{eq:extend-normal}) we obtain
\begin{multline*}
\eta_X\circ D^{\mu^F}_X(x)=\eta_X(\langle\{x\},\mu^F_1(0)\rangle)=F(\en_{\{x\}})(\langle 0,\bot\rangle)=\\
=\langle\en_{\{x\}}(0),(1+\en_{\{x\}})(\bot)\rangle=\langle x,\bot\rangle=\mu^F_X(x),
\end{multline*}
as required.
\end{proof}

The normal function $f$ from the introduction satisfies $f(\alpha)\leq(1+\alpha)^2$. We can now recover this result on the level of the prae-dilator $F$.

\begin{lemma}[$\rca_0$]
For each linear order $X$ we have an embedding of $D^F(X)$ into $(1+X)^2$, where the latter is equipped with the lexicographic order.
\end{lemma}
\begin{proof}
In view of the previous lemma it suffices to exhibit an embedding of $F(X)$ into $(1+X)^2$. Indeed, we have defined $F(X)\backslash\{\bot\}$ as a suborder of $(1+X)^2$. In order to obtain the desired embedding it suffices to map $\bot\in F(X)$ to the minimum element $\langle\bot,\bot\rangle\in(1+X)^2$. This is possible because $\langle\bot,\bot\rangle$ does not lie in the suborder $F(X)\backslash\{\bot\}$, due to the condition $y<_{1+X} x$ in Definition~\ref{def:F}.
\end{proof}

The following result concludes the reconstruction of $f$ in second order arithmetic:

\begin{corollary}[$\rca_0$]\label{cor:F-dil}
The normal prae-dilator $F$ is a normal dilator.
\end{corollary}
\begin{proof}
In view of Lemma~\ref{lem:F-prae-dil} it remains to show that $D^F(X)$ is well-founded for any well-order~$X$. By the previous lemma this reduces to the claim that $(1+X)^2$ is well-founded. More generally, the usual proof that any product $X\times Y$ of well-orders is well-founded goes through in $\rca_0$: Assume that there is a strictly decreasing sequence $(\langle x_n,y_n\rangle)_{n\in\mathbb N}$ in $X\times Y$. Then the sequence $(x_n)_{n\in\mathbb N}$ is non-increasing. Since $X$ is well-founded, there is an $N\in\mathbb N$ such that $x_n=x_N$ holds for all $n\geq N$ (otherwise a strictly decreasing sequence in $X$ could be constructed by recursion). Then $(y_n)_{n\geq N}$ is a strictly decreasing sequence in $Y$, which contradicts the assumption that $Y$ is well-founded.
\end{proof}

\section{From derivative to arithmetical comprehension}

In the present section we recall how derivatives of normal functions are defined in the context of second order arithmetic. We then show how the inequality $2^\alpha\leq f'(\alpha)$ from the introduction can be recovered in $\rca_0$. Finally, we conclude that the base theory in a result of Rathjen and the present author can be lowered from $\aca_0$ to $\rca_0$.

If $g'$ is the derivative of a normal function $g$, then we have $g\circ g'=g'$. To formulate this condition in second order arithmetic, we need to define the composition $T\circ S$ of normal prae-dilators. This is not entirely straightforward: In view of Definition~\ref{def:prae-dil} the orders $S(n)$ may be infinite, while $T$ is only defined on finite orders represented by natural numbers. In order to overcome this obstacle we use equation~(\ref{eq:extend-dil}) to extend $T$ beyond the category of natural numbers, and set
\begin{equation*}
(T\circ S)(n)=D^T(S(n)).
\end{equation*}
One can equip $T\circ S$ with the structure of a prae-dilator, as shown in~\cite[Section~2]{freund-rathjen_derivatives}. According to~\cite[Proposition~2.14]{freund-rathjen_derivatives} there is a family of isomorphisms
\begin{equation*}
\zeta^{T,S}_X:D^T\circ D^S(X)\xrightarrow{\cong} D^{T\circ S}(X).
\end{equation*}
If $S$ and $T$ are dilators, then equation (\ref{eq:induced-function}) yields
\begin{equation*}
f_{T\circ S}(\alpha)=\otp(D^{T\circ S}(\alpha))=\otp(D^T\circ D^S(\alpha))=\otp(D^T(f_S(\alpha)))=f_T\circ f_S(\alpha),
\end{equation*}
where the third equality relies on $D^S(\alpha)\cong\otp(D^S(\alpha))=f_S(\alpha)$ and the fact that~$D^T$ is functorial. Hence the given composition of dilators represents the usual composition of functions on the ordinals. If $T=(T,\mu^T)$ and $S=(S,\mu^S)$ are normal prae-dilators, then we can invoke equation~(\ref{eq:extend-normal}) to define $\mu^{T\circ S}_n:n\to(T\circ S)(n)$ by
\begin{equation*}
\mu^{T\circ S}_n=D^{\mu^T}_{S(n)}\circ\mu^S_n.
\end{equation*}
In~\cite[Lemma~2.16]{freund-rathjen_derivatives} it has been verified that this turns $T\circ S$ into a normal prae-dilator, and that we have
\begin{equation}\label{eq:mu-compose}
D^{\mu^{T\circ S}}_X=\zeta^{T,S}_X\circ D^{\mu^T}_{D^S(X)}\circ D^{\mu^S}_X.
\end{equation}
We can now recall the following notion, which has been introduced in~\cite{freund-rathjen_derivatives}:

\begin{definition}[$\rca_0$]
Let $T$ be a normal prae-dilator. An upper derivative of~$T$ consists of a normal prae-dilator $S$ and a natural transformation $\xi:T\circ S\Rightarrow S$ that satisfies $\xi\circ\mu^{T\circ S}=\mu^S$.
\end{definition}

According to~\cite[Lemma~2.19]{freund-rathjen_derivatives}, the natural transformation $\xi$ can be extended into a family of order embeddings $D^\xi_X:D^{T\circ S}(X)\to D^S(X)$ with
\begin{equation}\label{eq:extend-upper-deriv}
D^\xi_X\circ D^{\mu^{T\circ S}}_X=D^{\mu^S}_X.
\end{equation}
If $S$ is a dilator, then the embedding $D^\xi_\alpha$ witnesses
\begin{equation*}
f_T\circ f_S(\alpha)=\otp(D^{T\circ S}(\alpha))\leq\otp(D^S(\alpha))=f_S(\alpha),
\end{equation*} 
for any ordinal $\alpha$. The converse inequality is automatic when $f_T$ is a normal function. Hence $f_S$ does indeed enumerate fixed points of $f_T$. It is possible that some fixed points are omitted. In this case $f_S$ grows faster than the derivative of $f_T$, which justifies the term ``upper derivative". To characterize the actual derivative on the level of normal dilators one can consider initial objects in the category of upper derivatives, as shown in~\cite{freund-rathjen_derivatives}.

We can now state the main technical result of this paper. As explained in the introduction, the order $2^X$ consists of finite descending sequences with entries in~$X$.

\begin{theorem}[$\rca_0$]
Assume that $G$ and $\xi:F\circ G\Rightarrow G$ form an upper derivative of the normal dilator $F$ from Definition~\ref{def:F}. Then there is an order embedding of $2^X$ into~$D^G(X)$, for each linear order~$X$.
\end{theorem}
\begin{proof}
As preparation, we observe the following: By Lemma~\ref{lem:iso-reconstruct} and the results that we have recalled in the the first part of the present section, we get an embedding
\begin{equation*}
\xi^F_X:=D^\xi_X\circ\zeta^{F,G}_X\circ\eta_{D^G(X)}^{-1}:F(D^G(X))\to D^G(X).
\end{equation*}
According to Definition~\ref{def:normal-dil}, the normal prae-dilator $G$ comes with a natural transformation $\mu^G$. The latter extends into an embedding $D^{\mu^G}_X:X\to D^G(X)$, by equation~(\ref{eq:extend-normal}). The values of the desired embedding
\begin{equation*}
J:2^X\to D^G(X)
\end{equation*}
will be defined by recursion along sequences in $2^X$. To ensure that the recursion goes through we will simultaneously verify that we have
\begin{equation}\label{eq:embedding-sim-ind}
J(\langle x_1,\dots,x_n\rangle)<_{D^G(X)} D^{\mu^G}_X(x)\qquad\text{if we have $x_1<_X x$ or $n=0$}.
\end{equation}
Officially, the recursive construction and the inductive verification should be untangled. To show that this is possible, we point out that it will always be decidable whether the prerequisites of our recursive clauses are satisfied. Whenever they fail, we can thus assign $\xi^F_X(\bot)\in D^G(X)$ as a default value. Once the recursion is completed, the inductive verification shows that the default value is never required. Let us also point out that the induction statement has complexity $\Pi^0_1$ (note that (\ref{eq:embedding-sim-ind}) involves a universal quantification over~$x$). Since $\Sigma^0_1$-induction and $\Pi^0_1$-induction are equivalent (see e.g.~\cite[Corollary~II.3.10]{simpson09}), this shows that the induction can be carried out in~$\rca_0$. Let us now specify the details: For the base of the recursion we use the minimum element $\bot$ of $F(D^G(X))$ and set
\begin{equation*}
J(\langle\rangle)=\xi^F_X(\bot).
\end{equation*}
To verify condition~(\ref{eq:embedding-sim-ind}) we observe that equations~(\ref{eq:extend-upper-deriv}) and (\ref{eq:mu-compose}) and Lemma~\ref{lem:iso-reconstruct} yield
\begin{multline*}
D^{\mu^G}_X(x)=D^\xi_X\circ D^{\mu^{F\circ G}}_X(x)=D^\xi_X\circ\zeta^{F,G}_X\circ D^{\mu^F}_{D^G(X)}\circ D^{\mu^G}_X(x)=\\
=D^\xi_X\circ\zeta^{F,G}_X\circ\eta_{D^G(X)}^{-1}\circ\mu^F_{D^G(X)}\circ D^{\mu^G}_X(x)=\xi^F_X\circ\mu^F_{D^G(X)}\circ D^{\mu^G}_X(x)=\xi^F_X(\langle D^{\mu^G}_X(x),\bot\rangle).
\end{multline*}
In view of $\bot<_{F(D^G(X))}\langle D^{\mu^G}_X(x),\bot\rangle$ we get $J(\langle\rangle)<_{D^G(X)}D^{\mu^G}_X(x)$ for any $x\in X$, as required by condition~(\ref{eq:embedding-sim-ind}). In the recursion step we put
\begin{equation*}
J(\langle x_0,\dots,x_n\rangle)=\xi^F_X(\langle D^{\mu^G}_X(x_0),J(\langle x_1,\dots,x_n\rangle)\rangle).
\end{equation*}
To see that the argument $\langle D^{\mu^G}_X(x_0),J(\langle x_1,\dots,x_n\rangle)\rangle$ does indeed lie in the domain $F(D^G(X))$ of $\xi^F_X$, we must establish the condition $J(\langle x_1,\dots,x_n\rangle)<_{D^G(X)}D^{\mu^G}_X(x_0)$ from Definition~\ref{def:F}. By the definition of the order $2^X$ we have $x_1<_Xx_0$ or $n=0$. Hence the required inequality holds by condition~(\ref{eq:embedding-sim-ind}). Furthermore, condition~(\ref{eq:embedding-sim-ind}) remains valid in the recursion step: For $x_0<_X x$ we have $D^{\mu^G}_X(x_0)<_{D^G(X)}D^{\mu^G}_X(x)$. Together with the definition of $J(\langle x_0,\dots,x_n\rangle)$ and the equality $D^{\mu^G}_X(x)=\xi^F_X(\langle D^{\mu^G}_X(x),\bot\rangle)$ from above, this does indeed imply the condition $J(\langle x_0,\dots,x_n\rangle)<_{D^G(X)}D^{\mu^G}_X(x)$. It remains to show that $J$ is an order embedding. We establish
\begin{equation*}
\sigma<_{2^X}\tau\quad\Rightarrow\quad J(\sigma)<_{D^G(X)}J(\tau)
\end{equation*}
by joint induction on $\sigma$ and $\tau$ (or by induction on the length of $\tau$, which leads to an induction statement of complexity~$\Pi^0_1$). Let us first assume that we have
\begin{equation*}
\sigma=\langle\rangle<_{2^X}\langle y_0,\dots,y_m\rangle=\tau
\end{equation*}
with $\tau\neq\langle\rangle$. Since $\bot\in F(D^G(X))$ is minimal we do indeed get
\begin{equation*}
J(\sigma)=\xi^F_X(\bot)<_{D^G(X)}\xi^F_X(\langle D^{\mu^G}_X(y_0),J(\langle y_1,\dots,y_m\rangle)\rangle)=J(\tau).
\end{equation*}
Now consider an inequality
\begin{equation*}
\sigma=\langle x_0,\dots,x_n\rangle<_{2^X}\langle y_0,\dots,y_m\rangle=\tau.
\end{equation*}
We must either have $x_0<_X y_0$, or $x_0=y_0$ and $\langle x_1,\dots,x_n\rangle<_{2^X}\langle y_1,\dots,y_m\rangle$. If the latter holds, then we get $J(\langle x_1,\dots,x_n\rangle)<_{D^G(X)}J(\langle y_1,\dots,y_m\rangle)$ by the induction hypothesis. In either case we obtain
\begin{equation*}
\langle D^{\mu^G}_X(x_0),J(\langle x_1,\dots,x_n\rangle)\rangle<_{F(D^G(X))}\langle D^{\mu^G}_X(y_0),J(\langle y_1,\dots,y_m\rangle)\rangle.
\end{equation*}
By applying $\xi^F_X$ to both sides we get $J(\sigma)<_{D^G(X)}J(\tau)$.
\end{proof}

Recall that a (normal) prae-dilator $S$ is a dilator if and only if the order $D^S(X)$ is well-founded for any well-order~$X$. We can draw the following conclusion.

\begin{corollary}[$\rca_0$]\label{cor:deriv-to-aca}
Assume that any normal dilator $T$ has an upper derivative $(S,\xi)$ such that $S$ is a dilator. Then arithmetical comprehension holds.
\end{corollary}
\begin{proof}
In view of Theorem~\ref{thm:girard-hirst} it suffices to show that $2^X$ is well-founded for any given well-order~$X$. Construct $F$ as in Definition~\ref{def:F}. From Lemma~\ref{lem:F-prae-dil} and Corollary~\ref{cor:F-dil} we know that $F$ is a normal dilator. Hence the assumption of the present corollary yields an upper derivative $\xi:F\circ G\Rightarrow G$ such that $D^G(X)$ is well-founded. The previous theorem provides an order embedding of $2^X$ into $D^G(X)$, which witnesses that $2^X$ is well-founded as well.
\end{proof}

According to~\cite[Definition~2.26]{freund-rathjen_derivatives}, a derivative of a normal prae-dilator is an upper derivative that is initial in a suitable sense. In~\cite[Section~4]{freund-rathjen_derivatives} it has been shown how to construct a derivative $(\partial T,\xi^T)$ of a given normal prae-dilator $T$. The transformation of $T$ into $\partial T$ and $\xi^T$ can be implemented in $\rca_0$ (in particular it is computable). Hence $\rca_0$ proves that (upper) derivatives exist. What $\rca_0$ cannot show is that $X\mapsto D^{\partial T}(X)$ preserves well-foundedness when $X\mapsto D^T(X)$ does. Indeed, Rathjen and the present author have shown that the latter is equivalent to $\Pi^1_1$-bar induction (which asserts that $\Pi^1_1$-induction is available along any well-order). As explained in the introduction, we can now lower the base theory over which this equivalence holds (Theorem~5.9 of~\cite{freund-rathjen_derivatives} proves it over $\aca_0$).

\begin{corollary}[$\rca_0$]\label{cor:main-equiv}
The following are equivalent:
\begin{enumerate}
\item If $T$ is a normal dilator, then so is $\partial T$.
\item Any normal dilator~$T$ has an upper derivative $(S,\xi)$ such that $S$ is a dilator.
\item The principle of $\Pi^1_1$-bar induction holds.
\end{enumerate}
\end{corollary}
\begin{proof}
To see that (1) implies (2) it suffices to know that $\partial T$ and $\xi^T$ form an upper derivative of $T$. This holds by~\cite[Proposition~4.11]{freund-rathjen_derivatives}, which was proved in~$\rca_0$. The implication from~(2) to (3) holds over $\aca_0$, by the original proof of~\cite[Theorem~5.9]{freund-rathjen_derivatives}. Now Corollary~\ref{cor:deriv-to-aca} of the present paper tells us that (2) implies arithmetical comprehension, which means that all ingredients of the proof become available over~$\rca_0$. The implication from (3) to (1) holds by~\cite[Theorem~5.8]{freund-rathjen_derivatives}, which was established in $\rca_0$.
\end{proof}

Combining Corollaries~\ref{cor:main-equiv} and~\ref{cor:deriv-to-aca} yields a somewhat indirect proof that $\Pi^1_1$-bar induction implies arithmetical comprehension. In fact, the latter is equivalent to the weaker principle of arithmetical transfinite induction, as shown by J.~Hirst~\cite{hirst99}.

\section{Ordinal exponentiation as a derivative}

In the present section we show that the derivative of the normal function $f$ from the introduction is given by $f'(\alpha)=\omega^{\omega^\alpha}$. We also specify a normal function $g$ with $g(\alpha)\leq 1+\alpha\cdot 2$ and $g'(\alpha)=\omega^{1+\alpha}$. The relevance of this result has been discussed at the end of the introduction. In contrast to the previous sections, we do not aim to formalize this section in a weak base theory.

The normal function~$f$ was defined by $f(\alpha)=1+\sum_{\gamma<\alpha}(1+\gamma)$, or more formally by the recursive clauses
\begin{align*}
f(0)&=1,\\
f(\alpha+1)&=f(\alpha)+1+\alpha,\\
f(\lambda)&=\textstyle\sup_{\alpha<\lambda}f(\alpha)\qquad\text{for $\lambda$ limit}.
\end{align*}
Recall that $\alpha>0$ is multiplicatively (resp.~additively) principal if $\beta,\gamma<\alpha$ implies $\beta\cdot\gamma<\alpha$ (resp.~$\beta+\gamma<\alpha$). The following determines the derivative of $f$.

\begin{lemma}
 We have $f(\alpha)=\alpha$ if and only if $\alpha$ is a multiplicatively principal limit ordinal.
\end{lemma}
\begin{proof}
 Assume that $f(\alpha)=\alpha$ holds. In view of $f(1)>f(0)=1$ we get $\alpha>1$. By the definition of $f$ we also see that $0<\beta<\alpha$ implies
 \begin{equation*}
  \beta+1\leq f(\beta)+1<f(\beta)+1+\beta=f(\beta+1)\leq f(\alpha)=\alpha,
 \end{equation*}
 so that $\alpha$ is a limit. We can now infer that $\alpha$ is additively principal: Consider $\beta,\gamma<\alpha$ and set $\delta:=\max\{\beta,\gamma\}$. Since $\alpha$ is a limit, we get $\delta+1<\alpha$ and then
 \begin{equation*}
  \beta+\gamma\leq f(\delta)+1+\delta=f(\delta+1)<f(\alpha)=\alpha.
 \end{equation*}
 By a straightforward induction on $\gamma$ we get $\beta\cdot\gamma\leq f(\beta+\gamma)$. Since $\alpha$ is additively principal, it follows that $\beta,\gamma<\alpha$ implies
 \begin{equation*}
  \beta\cdot\gamma\leq f(\beta+\gamma)<f(\alpha)=\alpha.
 \end{equation*}
 Now assume that $\alpha$ is a multiplicatively (and hence additively) principal limit ordinal. Then $\gamma<\alpha$ implies $1+\gamma^2<\alpha$. In the introduction we have noted that $f(\gamma)$ is bounded by $1+\gamma^2$. Hence we get
 \begin{equation*}
  f(\alpha)=\textstyle\sup_{\gamma<\alpha}f(\gamma)\leq\textstyle\sup_{\gamma<\alpha}(1+\gamma^2)\leq\alpha.
 \end{equation*}
 The inequality $\alpha\leq f(\alpha)$ is automatic, since $f$ is strictly increasing.
\end{proof}

The derivative of $f$ can now be described as follows:

\begin{corollary}
 We have $f'(\alpha)=\omega^{\omega^\alpha}$ for any ordinal $\alpha$.
\end{corollary}
\begin{proof}
 It is known that an infinite ordinal is multiplicatively principal if and only if it is of the form $\omega^{\omega^\alpha}$ (see e.\,g.~\cite[Exercise~3.3.15]{pohlers-proof-theory}). Hence the previous lemma implies that $\alpha\mapsto\omega^{\omega^\alpha}$ is the increasing enumeration of the fixed points of $f$. The claim follows by the definition of the derivative.
\end{proof}

In the rest of this section we construct a normal function $g$ with $g'(\alpha)=\omega^{1+\alpha}$. Such a function can be defined by
\begin{align*}
 g(0)&=1,\\
 g(\alpha+1)&=(\alpha+1)\cdot 2,\\
 g(\lambda)&=\textstyle\sup_{\alpha<\lambda}g(\alpha)\quad\text{for $\lambda$ limit}.
\end{align*}
By induction on the limit ordinal $\lambda$ we get
\begin{equation*}
 g(\lambda)\leq\textstyle\sup_{\alpha<\lambda}\alpha\cdot2\leq\lambda\cdot 2.
\end{equation*}
In particular we have $g(\lambda)<g(\lambda+1)$, which readily implies that $g$ is strictly increasing. We also obtain $g(\alpha)\leq1+\alpha\cdot 2$ for any ordinal $\alpha$, as promised above. To characterize the derivative of $g$ we show the following:

\begin{lemma}
 We have $g(\alpha)=\alpha$ if and only if $\alpha$ is an additively principal limit ordinal.
\end{lemma}
\begin{proof}
 First assume that we have $g(\alpha)=\alpha$. In view of $g(0)=1$ we get $\alpha>0$. Since $g(\gamma+1)>\gamma+1$ holds for any successor, we learn that $\alpha$ must be a limit. In order to show that $\alpha$ is additively principal we consider arbitrary ordinals $\beta,\gamma<\alpha$. Setting $\delta:=\max\{\beta,\gamma\}$, we get
 \begin{equation*}
  \beta+\gamma<(\delta+1)\cdot 2=g(\delta+1)\leq g(\alpha)=\alpha.
 \end{equation*}
 Conversely, assume that $\alpha$ is an additively principal limit ordinal. Then $\gamma<\alpha$ implies~$\gamma\cdot 2<\alpha$, which yields
 \begin{equation*}
  g(\alpha)\leq\textstyle\sup_{\gamma<\alpha}\gamma\cdot 2\leq\alpha.
 \end{equation*}
 Yet again, the inequality $\alpha\leq g(\alpha)$ is automatic.
\end{proof}

We can now describe the derivative of $g$:

\begin{corollary}
 We have $g'(\alpha)=\omega^{1+\alpha}$ for any ordinal $\alpha$.
\end{corollary}
\begin{proof}
 It is well-known that an ordinal is additively principal if and only if it is of the form $\omega^\alpha$ (consider Cantor normal forms). Excluding $\omega^0=1$, we see that the additively principal limit ordinals are those of the form~$\omega^{1+\alpha}$. Now the claim follows by the previous lemma.
\end{proof}

To conclude, we explain why we have used $f$ rather than $g$ to lower the base theory of~\cite[Theorem~5.9]{freund-rathjen_derivatives}: In order to represent $g$ by a normal dilator we would need uniform notation systems for the values of this function. Elements of $g(\alpha+1)$ can be written as $\beta$ or $(\alpha+1)+\beta$ with $\beta<\alpha+1$, which suggests a relativized ordinal notation system. Canonical representations for elements of $g(\lambda)$ appear less obvious when $\lambda$ is a limit. For example, the ordinal $\omega+2\in g(\omega\cdot 2)$ could be written as $(\omega+1)+1\in g(\omega+1)$, as $(\omega+2)+0\in g(\omega+2)$ or as $\omega+2\in g(\omega+3)$. It would be interesting to know whether $g$ does have a reasonable representation as a normal dilator.

\bibliographystyle{amsplain}
\bibliography{Exp-deriv}

\newcommand{\noopsort}[1]{}
\providecommand{\bysame}{\leavevmode\hbox to3em{\hrulefill}\thinspace}
\providecommand{\MR}{\relax\ifhmode\unskip\space\fi MR }
\providecommand{\MRhref}[2]{%
  \href{http://www.ams.org/mathscinet-getitem?mr=#1}{#2}
}
\providecommand{\href}[2]{#2}
\begin{thebibliography}{10}

\bibitem{aczel-phd}
Peter Aczel, \emph{Mathematical {P}roblems in {L}ogic}, Ph{D} thesis, Oxford,
  1966.

\bibitem{aczel-normal-functors}
\bysame, \emph{Normal functors on linear orderings}, Journal of Symbolic Logic
  \textbf{32} (1967), p.~430, abstract to a paper presented at the annual
  meeting of the Association for Symbolic Logic, Houston, Texas, 1967.

\bibitem{freund-thesis}
Anton Freund, \emph{Type-{T}wo {W}ell-{O}rdering {P}rinciples, {A}dmissible
  {S}ets, and ${\Pi}^1_1$-{C}omprehension}, Ph{D} thesis, University of Leeds,
  2018, available via \url{http://etheses.whiterose.ac.uk/20929/}.

\bibitem{freund-computable}
\bysame, \emph{Computable aspects of the {B}achmann-{H}oward principle},
  Journal of Mathematical Logic \textbf{20} (2020), no.~2, article no.~2050006,
  26 pp.

\bibitem{freund-rathjen_derivatives}
Anton Freund and Michael Rathjen, \emph{Derivatives of normal functions in
  reverse mathematics}, Annals of Pure and Applied Logic \textbf{172} (2021),
  no.~2, article no.~102890, 49 pp.

\bibitem{girard-pi2}
Jean-Yves Girard, \emph{${\Pi^1_2}$-logic, part 1: Dilators}, Annals of Pure
  and Applied Logic \textbf{21} (1981), 75--219.

\bibitem{girard87}
\bysame, \emph{{P}roof {T}heory and {L}ogical {C}omplexity, {V}olume {I}},
  Studies in Proof Theory, Bibliopolis, Napoli, 1987.

\bibitem{hirst94}
Jeffry~L. Hirst, \emph{Reverse mathematics and ordinal exponentiation}, Annals
  of Pure and Applied Logic \textbf{66} (1994), 1--18.

\bibitem{hirst99}
\bysame, \emph{Ordinal inequalities, transfinite induction, and reverse
  mathematics}, The Journal of Symbolic Logic \textbf{64} (1999), no.~2,
  769--774.

\bibitem{maclane-working}
Saunders {Mac Lane}, \emph{Categories for the {W}orking {M}athematician},
  2\textsuperscript{nd} ed., Graduate Texts in Mathematics, vol.~5, Springer,
  1998.

\bibitem{marcone-montalban}
Alberto Marcone and Antonio Montalb{\'a}n, \emph{The {V}eblen functions for
  computability theorists}, Journal of Symbolic Logic \textbf{76} (2011),
  575--602.

\bibitem{pohlers-proof-theory}
Wolfram Pohlers, \emph{Proof {T}heory. {T}he {F}irst {S}tep into
  {I}mpredicativity}, Springer, Berlin, 2009.

\bibitem{rathjen-model-bi}
Michael Rathjen and Pedro Francisco~Valencia Vizca\'{i}no, \emph{Well ordering
  principles and bar induction}, Gentzen's centenary: The quest for consistency
  (Reinhard Kahle and Michael Rathjen, eds.), Springer, Berlin, 2015,
  pp.~533--561.

\bibitem{schuette77}
Kurt Sch{\"u}tte, \emph{Proof {T}heory}, Grundlehren der Mathematischen
  Wissenschaften, vol. 225, Springer, Berlin, 1977.

\bibitem{simpson09}
Stephen~G. Simpson, \emph{Subsystems of {S}econd {O}rder {A}rithmetic},
  Perspectives in Logic, Cambridge University Press, 2009.

\end{thebibliography}

\end{document}